\documentclass[12pt,reqno]{article}
\usepackage{amsmath,amsthm,amsfonts,amssymb,amscd}
\usepackage[pdftex]{graphicx}
\usepackage{psfrag}
\usepackage{textcomp}
\usepackage[pctex32]{color}

\theoremstyle{plain}
\newtheorem{thm}{Theorem}[section]

\newtheorem{prop}[thm]{Proposition}
\newtheorem{clly}[thm]{Corollary}
\newtheorem{lemma}[thm]{Lemma}
\newtheorem{defi}[thm]{Definition}

\newtheorem{maintheorem}{Theorem}

\newcommand{\re}{{\Bbb R}}

\title{Sectional Anosov flows: Existence of Venice masks with two singularities}
\author{A. M. L\'opez B., H. M. S\'anchez S.
        \thanks{
{\em Key words and phrases}:
Sectional Anosov flow, Maximal invariant, Sectional hyperbolic set, Transitive set, Venice mask.
This work is partially supported by CAPES, Brazil.}}
\date{}

\begin{document}
\maketitle

\begin{abstract}
We show the existence of venice masks \index{Venice mask}
(i.e. nontransitive sectional Anosov flows with dense periodic orbits,
\cite{bmp}, \cite{mp}, \cite{mp2},\cite{lec}) containing two equilibria on certain
compact $3$-manifolds. Indeed, we present two type of examples
in which the homoclinic classes composing their maximal invariant set
intersect in a very different way.
\end{abstract}


\section{Introduction}

The dynamical systems theory describes different properties about asymptotic behavior, 
stability, relationships among system's elements and its properties. 
It is well known that the hyperbolic systems own some features and 
properties that provide very important information about its behavior. 
With the purpose of extending the notion of hyperbolicity, arise definitions and 
a new theory, such as partial hyperbolicity, singular hyperbolicity and sectional hyperbolicity.
Thus, we begin by considering the relationship between the 
hyperbolic and sectional hyperbolic theory. Recall, the sectional hyperbolic sets and 
sectional Anosov flows were introduced in \cite{mo} and \cite{mem} respectively 
as a generalization of the hyperbolic sets and Anosov flows. They contain important examples such as 
the saddle-type hyperbolic attracting sets, the geometric and multidimensional 
Lorenz attractors \cite{abs}, \cite{bpv}, \cite{gw}.
A natural way is to observe the properties that are preserved or 
which are not in the new scenario.
Particularly, we mention two important properties related to hyperbolic sets which are not 
satisfied by all sectional hyperbolic sets. The first is the spectral decomposition theorem \cite{sma}. 
It says that an attracting hyperbolic set $\Lambda=Cl(Per(X))$ is a finite disjoint union of 
homoclinic classes, where $Per(X)$ is the set of periodic points of $X$. The second says that an Anosov 
flow on a closed manifold is transitive if and only if it has dense periodic orbits.
This results are false for sectional Anosov flows, i.e., sets whose maximal invariant 
is a sectional-hyperbolic set \cite{mp2}. Specifically, it is proved that there exists a 
sectional Anosov flow such that it is supported on a compact $3$-manifold, it has dense periodic 
orbits, is the union non disjoint of two homoclinic classes but is not transitive. So, a sectional 
Anosov flow is said a {\em Venice mask} if it has dense periodic orbits but is not transitive. 
The only known examples of venice masks have one or three singularities, 
and they are characterized by having two properties: are the union non 
disjoint of two homoclinic classes and the intersection of its homoclinic 
classes is the closure of the unstable manifold of a singularity \cite{mp}, \cite{mp2}, \cite{bmp}.
Thus, we provide two examples 
of venice masks with two singularities, but with different features. 
In particular, each one is the union of two different homoclinic classes. However, 
for the first, the intersection of homoclinic classes is 
the closure of the unstable manifold of two singularities. 
Whereas for the second, the intersection of homoclinic classes 
is just a hyperbolic periodic orbit. \\

Let us state our results in a more precise way.\\

Consider a Riemannian compact manifold $M$ of dimension three (a {\em compact $3$-manifold} for short).
		We denote by $\partial M$ the boundary of $M$.
		Let ${\cal X}^1(M)$ be the space
		of $C^1$ vector fields in $M$ endowed with the
		$C^1$ topology.
		Fix $X\in {\cal X}^1(M)$, inwardly
		transverse to the boundary $\partial M$ and denotes
		by $X_t$ the flow of $X$, $t\in I\!\! R$.
		
The {\em $\omega$-limit set} of $p\in M$ is the set
$\omega_X(p)$ formed by those $q\in M$ such that $q=\lim_{n \rightarrow \infty}X_{t_n}(p)$ for some
sequence $t_n\to\infty$. The {\em $\alpha$-limit set} of $p\in M$ is the set
$\alpha_X(p)$ formed by those $q\in M$ such that $q=\lim_{n \rightarrow \infty}X_{t_n}(p)$ for some
sequence $t_n\to -\infty$. The {\em non-wandering set} of $X$ is
the set $\Omega(X)$ of points $p\in M$ such that
for every neighborhood $U$ of $p$ and every $T>0$ there
is $t>T$ such that $X_t(U)\cap U\neq\emptyset$.
Given $\Lambda \in M$ compact, we say that $\Lambda$ is {\em invariant} 
if $X_t(\Lambda)=\Lambda$ for all $t\in I\!\! R$.
We also say that $\Lambda$ is {\em transitive} if
$\Lambda=\omega_X(p)$ for some $p\in \Lambda$; {\em singular} if it
contains a singularity and
{\em attracting}
if $\Lambda=\cap_{t>0}X_t(U)$
for some compact neighborhood $U$ of it.
This neighborhood is often called
{\em isolating block}.
It is well known that the isolating block $U$ can be chosen to be
positively invariant, i.e., $X_t(U)\subset U$ for all
$t>0$.
An {\em attractor} is a transitive attracting set.
An attractor is {\em nontrivial} if it is
not a closed orbit.

The {\em maximal invariant} set of $X$ is defined by
	$M(X)= \bigcap_{t \geq 0} X_t(M)$.

\begin{defi}
		\label{hyperbolic}
		A compact invariant set $\Lambda$ of $X$ is {\em hyperbolic}
		if there are a continuous tangent bundle invariant decomposition
		$T_{\Lambda}M=E^s\oplus E^X\oplus E^u$ and positive constants
		$C,\lambda$ such that

		\begin{itemize}
		\item $E^X$ is the vector field's
		direction over $\Lambda$.
		\item $E^s$ is {\em contracting}, i.e.,
		$
		\mid\mid DX_t(x) \left|_{E^s_x}\right.\mid\mid
		\leq Ce^{-\lambda t}$,
		for all $x \in \Lambda$ and $t>0$.
		\item $E^u$ is {\em expanding}, i.e.,
		$
		\mid\mid DX_{-t}(x) \left|_{E^u_x}\right.\mid\mid
		\leq Ce^{-\lambda t},
		$
		for all $x\in \Lambda$ and $t> 0$.
		\end{itemize}
\end{defi}

We denote by $m(L)$ the minimum co-norm of a linear
	operator $L$, i.e., $m(L)= inf_{v \neq 0} \frac{\left\|Lv\right\|}{\left\|v\right\|} $.
\begin{defi}
\label{d2}
A compact invariant set
$\Lambda$ of $X$
is {\em partially hyperbolic}
if there is a continuous invariant
splitting
$
T_\Lambda M=E^s\oplus E^c
$
with $E^c_x\neq 0$, $E^s_x\neq 0$ for all $x\in \Lambda$, such that the following properties
hold for some positive constants $C,\lambda$:

\begin{itemize}
\item
$E^s$ is {\em contracting}, i.e.,
$
\mid\mid DX_t(x) \left|_{E^s_x}\right. \mid\mid
\leq Ce^{-\lambda t},
$
for all $x\in \Lambda$ and $t>0$.
\item
$E^s$ {\em dominates} $E^c$, i.e.,
$
\frac{\mid\mid DX_t(x) \left|_{E^s_x}\right. \mid\mid}{m(DX_t(x) \left|_{E^c_x}\right. )}
\leq Ce^{-\lambda t},
$
for all $x\in \Lambda$ and $t>0$.
\end{itemize}
\end{defi}

We say the central subbundle $E^c_x$ of $\Lambda$ is
{\em sectionally-expanding} if
$$dim(E^c_x) = 2$$ and
$$\left| det(DX_t(x) \left|_{E^c_x}\right. ) \right| \geq C^{-1}e^{\lambda t},\quad
\forall x \in \Lambda \quad and \quad t > 0. $$
Here $det(DX_t(x) \left|_{E^c_x}\right. )$ denotes
the jacobian of $DX_t(x)$ along $E^c_x$.

\begin{defi}
\label{shs}
A {\em sectional-hyperbolic set}
is a partially hyperbolic set whose singularities (if any) are hyperbolic and whose central subbundle is sectionally-expanding.
\end{defi}

\begin{defi}
\label{secflow}
We say that $X$ is a {\em Anosov flow} if $M(X)=M$ is a hyperbolic
set. $X$ is a {\em sectional-Anosov flow} if $M(X)$ is a sectional-hyperbolic
set.
\end{defi}

The Invariant Manifold Theorem [3] asserts that if $x$ belongs
to a hyperbolic set $H$ of $X$, then the sets

$$W^{ss}_X(p)  =  \{x\in M:d(X_t(x),X_t(p))\to 0, t\to \infty\} \qquad and$$

$$W^{uu}_X(p)  =  \{x\in M:d(X_t(x),X_t(p))\to 0, t\to -\infty\},$$

are $C^1$ immersed submanifolds of $M$ which are tangent at $p$ to the subspaces $E^s_p$ and $E^u_p$ of $T_pM$ respectively.

$$W^{s}_X(p) =   \bigcup_{t\in I\!\! R}W^{ss}_X(X_t(p))\qquad\and\qquad W^{u}_X(p)  =   \bigcup_{t\in I\!\! R}W^{uu}_X(X_t(p))$$

are also $C^1$ immersed submanifolds tangent to $E^s_p\oplus E^X_p$ and $E^X_p\oplus E^u_p$ at $p$ respectively.

Recall that a singularity of a vector field is hyperbolic if
the eigenvalues of its linear part
have non zero real part.

\begin{defi}
\label{ll}
We say that a singularity $\sigma$ of a sectional-Anosov flow $X$ is {\em Lorenz-like}
if  it has three real eigenvalues $\lambda^{ss},\lambda^{s},\lambda^u$ with $\lambda^{ss}<\lambda^s<0<-\lambda^s<\lambda^u$.
$W^s(\sigma)$ is the manifold associated to eigenvalues $\lambda^{ss},\lambda^s$, 
and $W^{ss}(\sigma)$ is the manifold associated to eigenvalue $\lambda^{ss}$.
\end{defi}

\begin{defi}
A periodic orbit of $X$ is the orbit of some $p$ for which there is a minimal
$t > 0$ (called the period) such that $X_t(p) = p$.
\end{defi}

A homoclinic orbit of a hyperbolic periodic orbit $O$ is an orbit $\gamma\subset W^s(O)\cap W^u(O)$.
If additionally $T_qM = T_qW^s(O) + T_qW^u(O)$ for some (and hence all) point $q\in\gamma$, then we say that $\gamma$ is a transverse homoclinic orbit of $O$. The homoclinic class $H(O)$ of a hyperbolic periodic orbit $O$ is the closure of the union of the transverse homoclinic orbits of $O$. We say that a set $\Lambda$ is a homoclinic class if $\Lambda = H(O)$ for some hyperbolic periodic orbit $O$.

\begin{defi}
A Venice mask is a sectional-Anosov
flow with dense periodic
orbits which is not transitive.
\end{defi}

With these definitions we can state our main results.

\begin{maintheorem}
\label{thF}
There exists a Venice mask $X$ with two singularities supported on a $3$-manifold $M$, such that:
\begin{itemize}
\item $M(X)$ is the union of two homoclinic classes $\mathcal{H}_X^1, \mathcal{H}_X^2$.
\item $\mathcal{H}_X^1\cap \mathcal{H}_X^2=O$, where $O$ is a hyperbolic periodic orbit.
\end{itemize}
\end{maintheorem}

\begin{maintheorem}
\label{thG}
There exists a Venice mask $Y$ with two singularities supported on a $3$-manifold $N$, such that:
\begin{itemize}
\item $N(Y)$ is the union of two homoclinic classes $\mathcal{H}_Y^1, \mathcal{H}_Y^2$.
\item $\mathcal{H}_Y^1\cap \mathcal{H}_Y^2=Cl(W^u(\sigma_1)\cup W^u(\sigma_2))$, 
      where $\sigma_1,\sigma_2$ are the singularities of $Y$.
\end{itemize}
\end{maintheorem}

In section \ref{pre}, we shall be described 
briefly this construction by using one-dimensional 
and two-dimensional maps. In section \ref{exm1}, 
from modifications on the previous maps in Section \ref{pre} and 
by considering  a plug, we shall prove the Theorem \ref{thF}. 
In the same way, in Section \ref{exm2}, by using the venice mask with a 
unique singularity, the Theorem \ref{thG} will be obtained 
by gluing a particular plug preserving the original flow.

\section{Preliminaries}

\subsection{Original plugs}
  In order to obtain the three-dimensional vector field of our 
example, we begin by considering the well known {\em Plykin attractor}\index{Plykin attractor}
and the {\em Cherry flow}\index{Cherry flow}\index{Flow!Cherry} ( See \cite{rob}, \cite{dmp}).

We give a sketch of the flow construction. It will be constructed through 
three steps, firstly by modifying the Cherry flow. 
In fact, we consider a vector field in the square whose flow is
described in Figure \ref{4} a). Note that this vector field has two equilibria: 
a saddle $\sigma$ and a sink $p$. For $\sigma$ one has that its eigenvalues 
$\{\lambda_s, \lambda_u\}$ of $\sigma$ satisfy the relation

$$\lambda_s < 0 < -\lambda_s < \lambda_u.$$

We have depicted a small disk $D$ centered at the attracting equilibrium $p$
Figure \ref{4} b). Note that the flow is pointing inward the edge of the disk.
This finishes the first step for the construction. 

\begin{figure}[htv]
\begin{center}
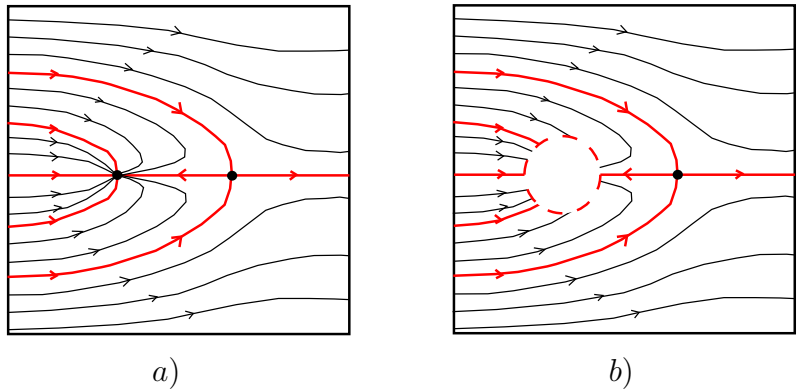
\caption{\label{4} Cherry flow.}
\end{center}
\end{figure}
\pagebreak
\newpage

For the second step we multiply the above vector field by a strong contraction
$\lambda_{ss}$ in order to obtain the vector field described in Figure \ref{c4} a). 
We can choose $\lambda_{ss}$ such that $-\lambda_{ss}$ be large, 
so the resulting vector field will have a Lorenz-like singularity
and this new eigenvalue will be associated with the strong manifold of the 
singularity. This yields a Cherry flow box \index{Flow!Cherry box}\index{Cherry flow!Box}
and finishes the second step for the construction.

\begin{figure}[htv]
\begin{center}
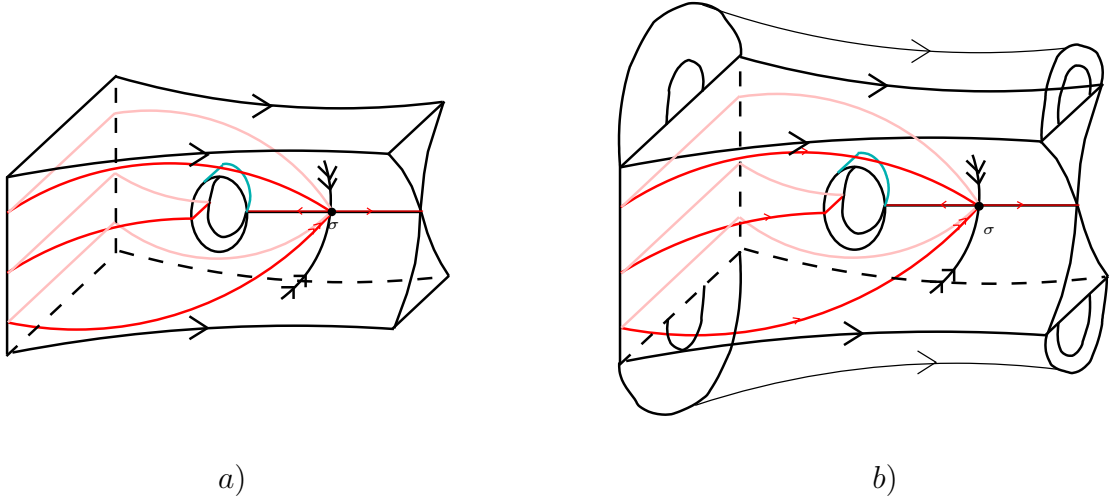
\caption{\label{c4} Cherry flow box and Plug \ref{c4}.}
\end{center}
\end{figure}


From Plykin attractor follows that the construction must have 
at least two holes inasmuch as we will use certain return map. 
Then, the final step is to glue two handles that provides two holes and the 
three dimensional vector field above in order to obtain the vector 
field whose flow is given in Figure \ref{c4} b). 
Hereafter the resulting vector field will be called of Plug \ref{c4}.\\

The hole indicated in this Figure \ref{c4} is nothing but the disk 
$D$ times a compact interval $I_1$. Again, note that the flow is pointing inward 
the edge of the hole by construction. For this reason, we take a solid $3$-ball and 
we define a flow on this one. Indeed, it flow has no singularities, it acts as 
in Figure \ref{5} and will be used for to glue the hole's bound with this one. 
Hereafter the resulting vector field will be called of Plug \ref{5}.

\begin{figure}[htv]
\begin{center}
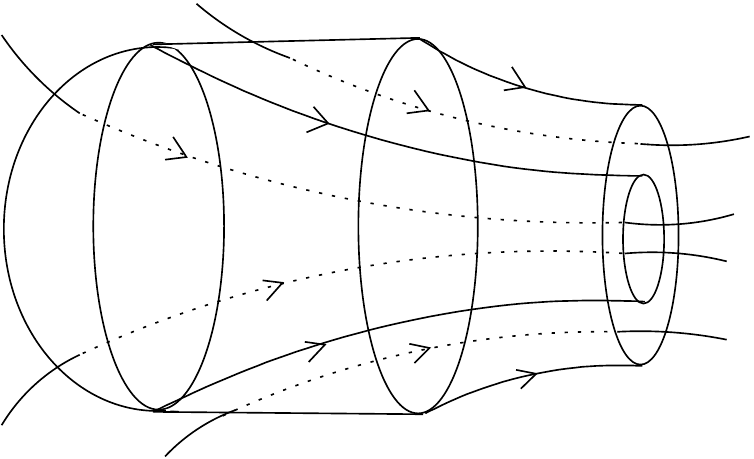
\caption{\label{5} Plug \ref{5}.}
\end{center}
\end{figure}
\pagebreak
\newpage

\section{Modified maps}
\label{maps}

We begin by considering the construction made in \cite{bmp} like model in order 
of obtain the vector fields $X$ and $Y$ of the main theorems. Recall that the 
original model provides tools for a three dimensional example with a unique singularity. 
The aim main is modify the original maps, in order to make a suspension of the 
modify maps via the new plugs. For this purpose, we will do 
such modifications followed by its original maps.

\subsection{One-dimensional map}
\label{map1}

Thus, in the same way of \cite{bmp}, we consider the 
branched $1$-manifold $\mathcal{B}$ consisting of a compact
interval and a circle with branch point $b$. We cut $\mathcal{B}$ open 
along $b$ to obtain a compact interval which we assume to be $[0, 1]$
for simplicity. In $[0, 1]$ we consider three points $0 < d_1 < d_{\ast} < d_2 < 1$,
where $d_{\ast}$ is depicted also in the Figure \ref{1}. These will be the discontinuity points of $f$
as a map of $[0,1]$. The set $\mathcal{B}\setminus\{d_{\ast}\}$ will be the domain of $f$.
We define $f:\mathcal{B}\setminus\{d_{\ast}\} \to \mathcal{B}$ in a way that its graph in $[0,1]$
is the one in Figure \ref{1}.

\begin{figure}[htv]
\begin{center}
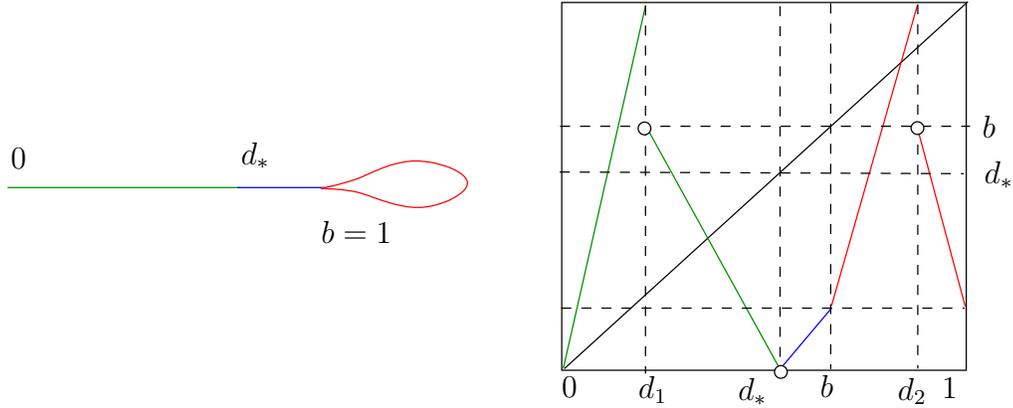
\caption{\label{1}The quotient space and  one-dimensional map.}
\end{center}
\end{figure}

By construction one has that $f$ satisfies the following hypotheses:

\textbf{(H1)}: $Dom(f) = [0, 1] \setminus\{d_{\ast}\}$.

\textbf{(H2)}: $f(0) = 0$; $f(d_1) = f(d_2) = 1$; $f(1) = f(b)\in (0, d_1)$.

\textbf{(H3)}: $f(d_1+) = f(d_2+) = b$; $f(d_1-) = f(d_2-) = 1;$ $f(d_{\ast}+) = f(d_{\ast}-) = 0$.

\textbf{(H4)}: $f([0, d_1]) = [0, 1]$; $f((d_1, d_{\ast})) = (0, b)$; $f((d_{\ast}, d_2]) = (0, 1]$; $f((d_2, 1]) =[f(b),b)$.

\textbf{(H5)}: $f$ is expanding, i.e., $f$ is $C^1$ in $Dom(f)$ and there is $\lambda > 1$ such that
 $|f'(x) |\geq\lambda$, for each $x \in Dom(f)$.

\subsection{Modified one-dimensional map}
\label{pre}

We realize a modification of the above map $f$ . Denote $d_{\ast}=d^+$ 
and let $f^+:\mathcal{B}^+\setminus\{d^+\}\to \mathcal{B}^+$ be in 
a way that its graph in $[0, 1]$ is the one in Figure \ref{6}.

Here, there exist $\varepsilon> 0$ small such that 
$\int_0^{d_1} \sqrt{[(f)'(x)]^2+1}dx<\int_0^{d_1} 
\sqrt{[(f^+)'(x)]^2+1}dx<\int_0^{d_1} \sqrt{[(f)'(x)]^2+1}dx+\varepsilon$ 
and $\int_{d_{\ast}}^{b} \sqrt{[(f)'(x)]^2+1}dx<\int_{d_{\ast}}^{b} 
\sqrt{[(f^+)'(x)]^2+1}dx<\int_{d_{\ast}}^{b} \sqrt{[(f)'(x)]^2+1}dx+\varepsilon$. 
Moreover $f^+$ satisfies \textbf{(H1)-(H5)}. 
We define $f^-(x)=f(-x)$ and denote 
$-d^+=d^-$. $f^-:\mathcal{B}^-\setminus\{d^-\}\to \mathcal{B}^-$.

\begin{figure}[htv]
\begin{center}
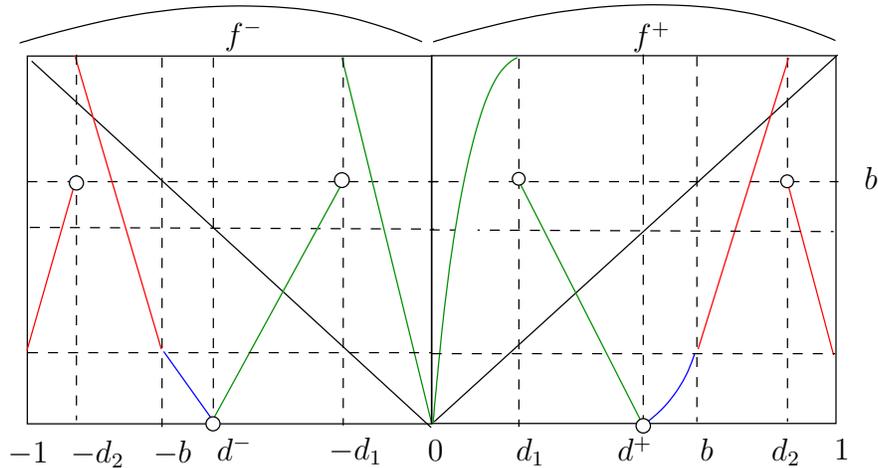
\caption{\label{6} Modified one-dimensional map.}
\end{center}
\end{figure}

These following results examining the properties of $f$ and 
appears in \cite{bmp}. This in turns through a structure closely 
related to \cite{bmp} and by construction we obtain the same properties 
for the $f^+$ map.

\begin{defi}
We say that $f$ is {\em locally eventually onto (leo for short)} \index{Locally eventually onto}
\index{Function!{\em leo}} if given any
open interval $I\subset [0,1]$ there is $m\geq 0$ such that $f^m(I) = [0,1]$.
\end{defi}

\begin{thm}
$f^+$ is {\em leo}.
\end{thm}

\begin{clly}
\label{clly}
The periodic points of $f^+$ are dense in $\mathcal{B}$. If $x\in \mathcal{B}$, then

$$\mathcal{B}=Cl\left(\bigcup_{n\geq 0}(f^+)^{-n}(x)\right).$$

\end{clly}

\subsection{Two-dimensional map}
\label{map2}

We consider the twice punctured planar region $R$ depicted in Figure \ref{2}. 
It is formed by: two half-annuli $A$, $F$, and four rectangles $B,C,D,E$. 
There is a middle vertical line denoted by $l$. Note that $l$ defines a 
plane reflexion throughout denoted by $\theta$. We assume 
$\theta(D) = C$, $\theta(E) = B$ and $\theta(F) = A$. In particular,
$\theta(R) = R$ and $\theta(d^+) = d^-$, where the vertical segments $d^-, d^+$ correspond to the
right-hand and left-hand boundary curves of $B$ and $D$ respectively. We define
$H^{-} = A \cup B \cup C$ and $H^{+} = D \cup E \cup F$.

%
\begin{figure}[htv]
\begin{center}
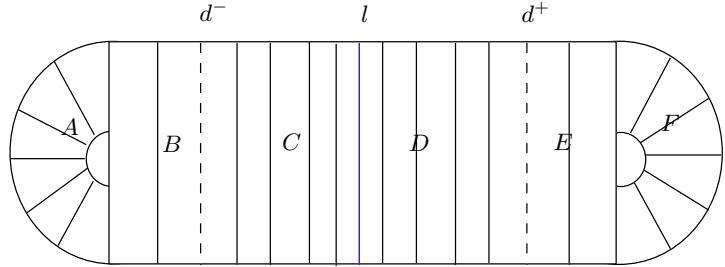
\caption{\label{2}Region $R$.}
\end{center}
\end{figure}

In \cite{bmp} was defined the $C^{\infty}$ map $G:R\setminus\{d^-, d^+\}\to Int(R)$. It satisfies the following hypotheses:
 
\textbf{(G1)}: $G$ and $\theta$ commute, i.e., $G \circ\theta = \theta\circ G$.
 
\textbf{(G2)}: $G$ preserves and contracts the foliation $\mathcal{F}$.
 
\textbf{(G3)}: Let $g : K \setminus\{d^-, d^+\} \to K$ be the map induced by $G$ in the leaf space $K$.
 
Then, the map $f^+$ defined by $f^+ = g|_{\mathcal{B}^+}$ satisfies the hypotheses \textbf{(H1)-(H5)},
with $f = f^+$, $\mathcal{B} = \mathcal{B}^+$ and $d_{\ast} = d^+$.\\
 

 \textbf{Properties of G}
 \begin{itemize}
 \item  By \textbf{(G1)}, $H^{+}$ and $H^{-}$ are invariant under $G$.
 \item Since $G$ contracts $\mathcal{F}$ (\textbf{(G2)}) we have that $W^s(x,G)$ 
 is union of leaves of $\mathcal{F}$. It follows from \textbf{(G2)}, \textbf{(G3)} and the 
 expansiveness in \textbf{(H5)} that all periodic points of $G$ are hyperbolic saddles.
 \item By \textbf{(G1)} we have that $G(l)\subset l$ and so $G$ has a fixed point $P$ in $l$. Clearly one has $\pi(P) = 0$.
 \end{itemize}
 
 Define
 
 $$A_G^-=Cl\left(\bigcap_{n\geq 1} G^n(H^{-}) \right), \qquad A_G^+=Cl\left(\bigcap_{n\geq 1} G^n(H^{+}) \right).$$

 \begin{thm}
 $A_G^-$ and $A_G^+$ are homoclinic classes and $P\in A_G^+\cap A_G^-$.
 \end{thm} 

\subsection{Modified two-dimensional map}
\label{pre1}


For the region $R$ in Figure \ref{2}, 
we define the $C^{\infty}$ map $H:R\setminus\{d^-, d^+\}\to Int(R)$ 
in a way that its image is as indicated in Figure \ref{7}. 
We require the following hypotheses:

\textbf{(L1)}: $H^-$, $H^+$ are invariant under $H$. $H(H^-\setminus\{d^{-}\})\subset H^-$ and $H(H^-\setminus\{d^{+}\})\subset H^+$.

\textbf{(L2)}: $H$ preserves and contracts the foliation $\mathcal{F}$.

\textbf{(L3)}: Let $h : K \setminus\{d^-, d^+\} \to K$ be the map induced by $H$ in the leaf space $K$.

Then, the map $f^{+(-)}$ defined by $f^{+(-)} = h|_{\mathcal{B}^{+(-)}}$ satisfies the hypotheses \textbf{(H1)-(H5)},
$\mathcal{B} = \mathcal{B}^{+(-)}$ and $d_{\ast} = d^{+(-)}$.

\begin{figure}[htv]
\begin{center}
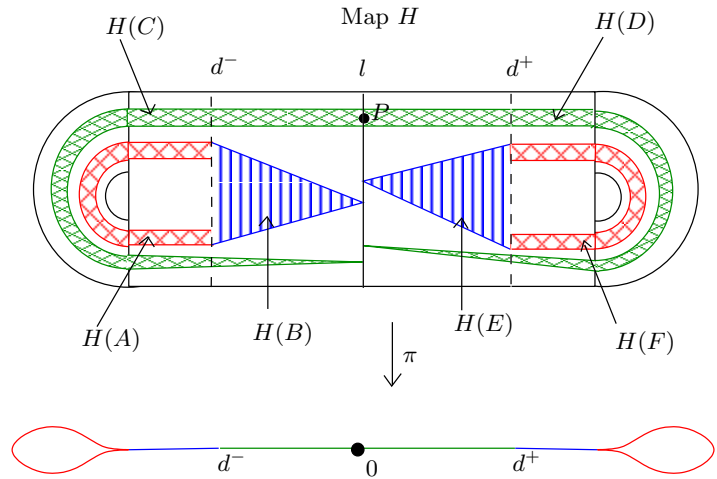
\caption{\label{7} The quotient space and modified two-dimensional map.}
\end{center}
\end{figure}

We observe that \textbf{(L1)} implies $H(l)\subset l$ and by contraction, $H$ has a fixed point $P\in l$. Again, for

$$A_H^-=Cl\left(\bigcap_{n\geq 1} H^n(H^{-}) \right), \qquad A_H^+=Cl\left(\bigcap_{n\geq 1} H^n(H^{+}) \right)$$

we have that $A_H^+$ and $A_H^-$ are homoclinic classes and $\{P\}= A_H^+\cap A_H^-$.

\pagebreak
\newpage

\subsection{Venice mask with one singularity}

Recall, by considering the original maps (Subsection \ref{map1}, \ref{map2}), 
and by using the plugs \ref{c4}, \ref{5}, in \cite{bmp} was construct the 
venice mask example with one singularity. Here, we provides a graphic idea 
in order to compare it with the new examples. 

The Figure \ref{ot} $a)$ shows the flow, 
whereas the Figure \ref{ot} $b)$ shows the ambient manifold that supports this one.
The ambient manifold is a solid bi-torus excluding two tori neighborhoods $V_1,V_2$ associated to
two repelling periodic orbits $O_1,O_2$ respectively.

\begin{figure}[htv]
\begin{center}
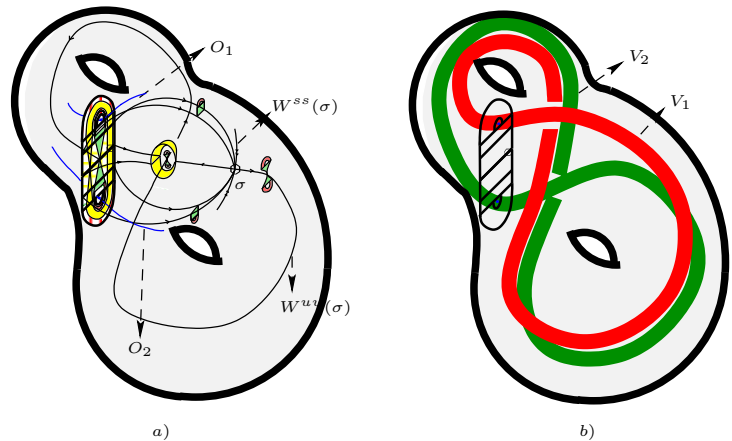
\caption{\label{ot}Venice mask with one singularity}
\end{center}
\end{figure}
\pagebreak
\newpage

\section{Venice mask's examples with two-singularities}

\subsection{Vector field $X$ and Example $1$.}
\label{exm1}

In this section, we construct a vector field 
$X$ which will satisfy the properties in the Theorem \ref{thF} by using
the subsection \ref{pre} and \ref{pre1}. 

We begin by considering a vector field as the Cherry flow 
described in Figure \ref{4}, with the same conditions of subsection \ref{pre}.

We called this flow of $A$ and we proceed to perturbe it, following the ideas of 
the well known DA-Attractor introduced by Smale (see \cite{rob}). 
Let $U$ be a neighborhood (relatively small) 
of $\sigma$. We can obtain a flow $\varphi^t$ such that 
$supp(\varphi^t-id)\subset U$ (Figure \ref{8} $a)$). Also, the derivate of the flow at 
$\sigma$ with respect to canonical basis in $T_{\sigma}Q$ is

\begin{align*}
D\varphi_{\sigma}^t=
\left(
\begin{array}{cc}
1 & 0 \\
0 & e^t
\end{array}
\right).
\end{align*}

We deform such a flow in order to obtain a one-parameter family of 
flows $B^t=\varphi^t\circ A$. Let $\tau>0$ such that $e^{\tau}\lambda_s>1$, 
so $\sigma$ is a source for $B^{\tau}$. Moreover, the new map has three 
fixed points on $W_X^s(\sigma)$, $\sigma$ a source and $\sigma_1$, $\sigma_2$ 
saddles. Moreover, there exists a neighborhood $V$ of $\sigma$ (not 
containing $\sigma_1$ and $\sigma_2$) contained in $U$ such 
that $B^{\tau}_s(V)\supset V$ for all $s>0$ (Figure \ref{8} $b)$).
Thus, we obtain a vector field as the square $Q$ whose flow $A$ is
described in Figure \ref{8}.

%
\begin{figure}[htv]
\begin{center}
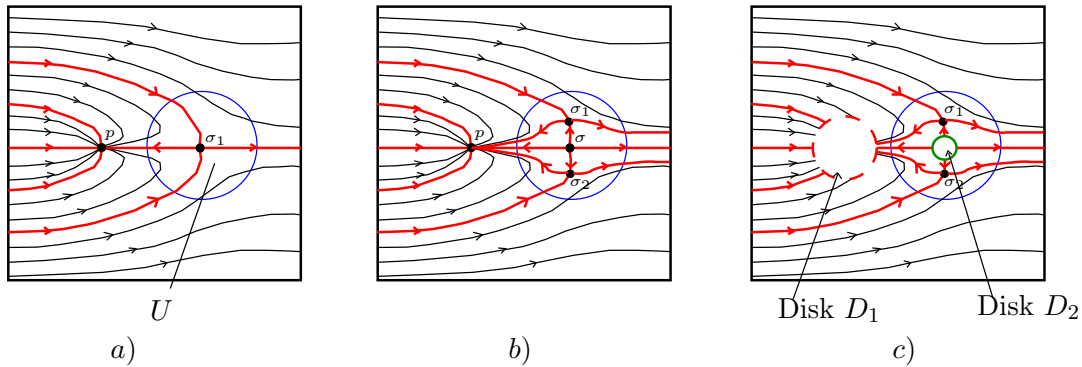
\caption{\label{8} Perturbed Cherry flow}
\end{center}
\end{figure}

Now, we remove two small disks $D_1$, $D_2=V$ centered at the 
attracting equilibrium $p$ and at the repelling equilibrium $\sigma$ 
respectively (Figure \ref{8} $c)$).


In the next step, we multiply the above vector field by a strong 
contraction $\lambda_{ss}$ in order to obtain the similar vector field 
described in Figure \ref{c4} $b)$. We choose 
$\lambda_{ss}$ such that $\sigma_1$ and $\sigma_2$ are Lorenz-like.

Now, we consider an interval $I_0=I_1 \times  \{p_0\}$, where $p_0$ is the 
point of intersection between $W_X^u(\sigma)$ and the disk 
$D_1$. We realize a modification in the flow 
such that a branched of $W_X^u(\sigma_1)$ intersects a connected 
component of $I_0\setminus\{p_0\}$ and a branched of $W_X^u(\sigma_2)$ 
intersects the other connected component of $I_0\setminus\{p_0\}$ (See \ref{px}).

The final step is to glue two handles on the $3$-dimensional 
vector field above in order to obtain the vector field whose 
flow is given in Figure \ref{px} $a)$. The resulting vector field
is what we shall call Plug $X$.

\begin{figure}[htv]
\begin{center}
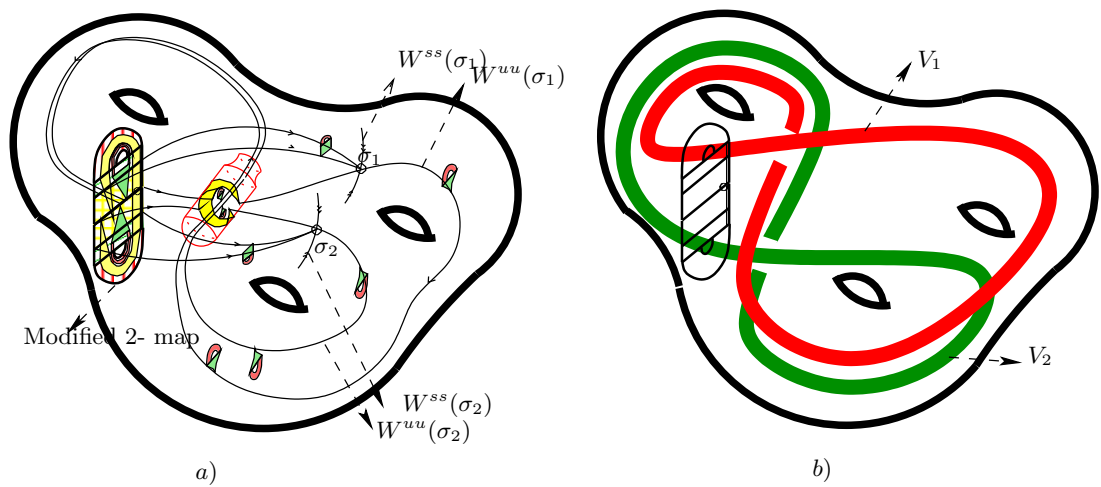
\caption{\label{px} Plug $X$ and its associated manifold.}
\end{center}
\end{figure}

In the same way as in Figure \ref{c4}, in this case, by multiplying the above vector field by a strong 
contraction generate two holes and it is nothing but the disks $D_1$ times a compact interval 
$I_1$, and $D_2$ times a compact interval $I_2$. Also, let us to use the Plug \ref{5} and 
apply on the hole associated to $D_1$. Note that the interval $I_2$ is chosen such that 
$D_2\times I_2$ produces the third hole on the ambient manifold. It generates a solid tritorus
(see Figure \ref{px} $b)$).

Then, we construct a vector field $X$ on a solid tritorus $ST_1$ in a
way that $X_t(ST_1)\subset Int(ST_1)$ for all $t > 0$ and $X$ is transverse to the boundary 
of the solid tritorus. The flow is obtained gluing plugs $X$ and \ref{5} as indicated in
Figure \ref{px} $a)$.

We require the following hypotheses:

\textbf{(X1)}: There are two repelling periodic orbits $O_1$, $O_2$ in $Int(ST_1)$ crossing the
holes of $R$.

\textbf{(X2)}: There are two solid tori neighborhoods $V_1, V_2 \subset Int(ST_1)$ of $O_1,O_2$ with
boundaries transverse to $X_{t}$ such that if $M = ST_1 \setminus (V_1 \cup V_2)$, then $M$ is a
compact neighborhood with smooth boundary transverse to $X_{t}$ and $X_{t}(M)\subset
M$ for $t>0$. As $M$ is a solid tritorus with two solid tori removed, we have that
$M$ is connected as indicated in Figure \ref{px} $b)$.

\textbf{(X3)}: $R \subset M$ and the return map $H$ induced by $X$ in $R$ satisfies the properties
\textbf{(L1)-(L3)} in Section \ref{pre1}. Moreover,

$$\{q\in M : X_{t}(q) \notin R, \forall t\in \re\} = \{\sigma_1,\sigma_2\}.$$

\bigskip

Now, define
$$A^+=Cl\left(\bigcup_{t\in\re} X_{t}(A^+_H)\right)\qquad\text{and}\qquad A^-=Cl\left(\bigcup_{t\in\re} X_{t}(A^-_H)\right).$$

\bigskip

\begin{prop}
\label{prop1}
$W_X^u(\sigma_1)\subset A^+$ and $W_X^u(\sigma_2)\subset A^-$.
\end{prop}

\begin{proof}
If $x\in H^+$ is a periodic point of $H$, then $G^n(x)\in R$ for all $n\leq 0$ and so
$x\in A_H^+ = Cl(\bigcap_{n\geq1} H^n(H^+))$. Therefore $x\in A^+$ (for $A_{H^+}\subset A^+$)
and by invariance of $A^{+}$, the full orbit of $x$ is contained in $A^+$. 

Second, the periodic points of $f^+$ in \textbf{(L3)} 
are dense in $\mathcal{B}$ by Corollary \ref{clly}. Then,
the periodic points of $H$ accumulate on $d^+$ in both 
connected components of $H^+\setminus d^+$. Since $d^+$ is contained 
in $W_X^s(\sigma_1)$, the full $X_{t}$-orbit of the periodic points of $H$ 
accumulating $d^+$ also accumulate on $W_X^u(\sigma_1)$. 
Then $W_X^u(\sigma_1)\subset A^+$ because $A^+$ is closed. Analogously, we have $W_X^u(\sigma_2)\subset A^-$.
\end{proof}

Define $A_H= A_H^+ \cup A_H^-$ and 

$$A=Cl\left(\bigcup_{t\in\re} X_{t}(A_H)\right),$$

\begin{lemma}
\label{homcla}
$A^+$ and $A^-$ are homoclinic classes of $X$ and  $A=A^+\cup A^-$.
\end{lemma}

\bigskip

\begin{proof}
See \cite{bmp}.
\end{proof}

\bigskip

\begin{prop}
\label{secans}
$X$ is a sectional Anosov flow.
\end{prop}

\begin{proof}

In the same way of \cite{bmp}, we will prove that $A$ is a sectional-hyperbolic set and 
$M(X)=A$. Indeed, how $A=A_1\cup A_2$ is union of homoclinic 
classes then $A$ has dense periodic orbits (Birkhoff-Smale Theorem). 
Moreover, of the hypotheses \textbf{(L2)} and \textbf{(L3)} 
follows that every periodic orbit of $X$ contained in $A$ has a 
hyperbolic splitting $T_OM=E_O^s\oplus E_O^X\oplus E_O^u$. 
Here, $E_O^s$ is due to \textbf{(L2)}, $E_O^u$ by \textbf{(L3)} 
and $E_O^X$ is the one-dimensional subbundle over $O$ induced by $X$. 
Let $Per(A)$ be the union of the periodic orbits of $X$ 
contained in $A$. Define the splitting

$$T_{Per(A)}M=F^s_{Per(A)}\oplus F^c_{Per(A)},$$

where $F^s_x=E^s_x$ and $F^c_x=E^X_x\oplus E^u_x$ for $x\in Per(A)$. As every periodic orbit in $M$
of every vector field $C^1$ close to $X$ is hyperbolic of saddle type, we can use the
arguments in \cite{mpp2} to prove that the splitting $T_{Per(A)}M = F^s_{Per(A)}\oplus F^c_{Per(A)}$ over
$Per(A)$ extends to a sectional-hyperbolic splitting $T_AM = F^s_A\oplus F^c_A$
over the whole $A = Cl(Per(A))$.

We conclude that $X$  is a sectional Anosov flow on $M$.

\end{proof}

\bigskip

\textbf{Proof of Theorem \ref{thF}}.

By using the Lemma \ref{homcla} and the Proposition \ref{secans} we have
that $X$ is a sectional Anosov flow and $M(X)$ is the union of two
homoclinic classes $H_X^1, H_X^2$, where $H_X^1=A^+$ and $H_X^2=A^-$.
Since $\{P\} = A^+_H\cap A^-_H$, it implies that  $H_X^1\cap H_X^2=O$, with $O$
the orbit associated to $P$. In particular $X$ is a Venice mask,
and by construction it has two singularities.

\subsection{Vector field $Y$ and Example $2$.}
\label{exm2}

In this section, we construct a vector field 
$Y$ which will satisfy the properties in the Theorem \ref{thG} by using
the results from \cite{bmp}. 

Firstly, in order to obtain the vector field $Y$, we begin by considering
the venice mask with one singularity. Unlike the previous section, in this case 
we will not perturb the flow. Moreover, we will change the flow by preserving the plugs 
\ref{c4}, \ref{5} and  we will remove a connected component of the flow 
and its ambient manifold.


The main aim of remove a connected component will be 
glue a new plug with different features, properties and 
that provides other singularity. This process is done in simple steps.
(see Figure \ref{steps}). Indeed, the important steps are Figure \ref{steps} $c)$, $d)$ 
and since we want a plug by containing a singularity, 
we will see that the this one has a hole,  which is produced by the singularity.

\begin{figure}[htv]
\begin{center}
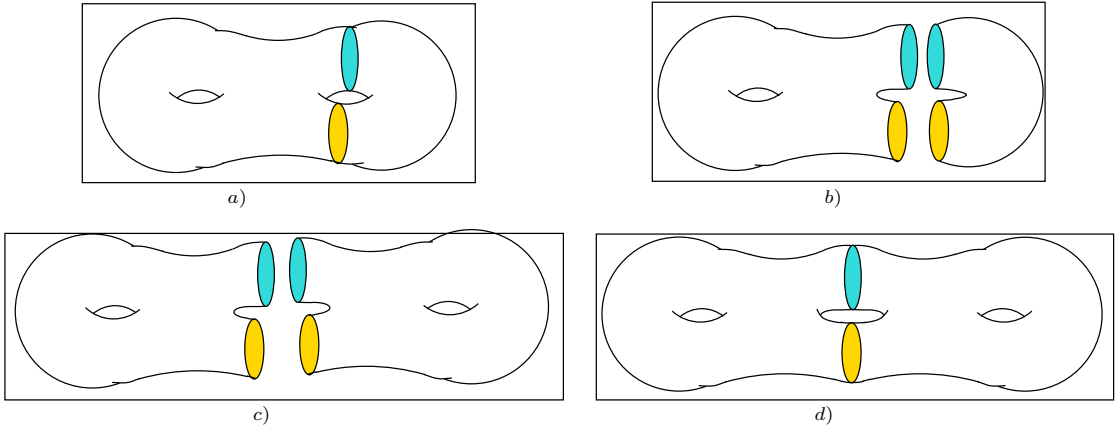
\caption{\label{steps} Steps by gluing the new plug.}
\end{center}
\end{figure}

\subsection{Flow through of the faces}

We begin by considering the plug \ref{c4} described in 
Figure \ref{c4} with the same conditions of subsections \ref{map1}, \ref{map2}.

For this purpose we need to observe with detail the flow behavior 
through of the faces removed. Indeed, we observe the vector field in the square whose flow is
described in Figure \ref{c4}.

Thus, it will be constructed the new plug through 
two steps. Firstly, we will be depicted a circle that represents the face $1$ on the Cherry flow 
and let us to study the flow behavior. It should be noted that this vector field exhibits two 
leaves which belong to the region $R$ and converge to the singularity, i.e., the region 
$R$ exhibits two singular leaves. Note that these leaves are crossing outward to the face $1$. 
In addition, note that there are trajectories crossing inward to the face $1$ too, such as 
the branch unstable manifold of the singularity. 
This shows that extensive analysis is necessary for understand the flow behavior to the face $1$.

We can observe that the top and bottom region 
of the singular leaves saturated by the flow are crossing through the face $1$, 
i.e., the flow is pointing outward of the face $1$.

By studying the complement of these regions, we have that the  
behavior of the leaves is depicted as Figure \ref{f1}. Here, 
this region exhibits two tangent leaves, whereas the other 
leaves intersect the region twice, i.e., the other leaves 
cross and return. 

\begin{figure}[htv]
\begin{center}
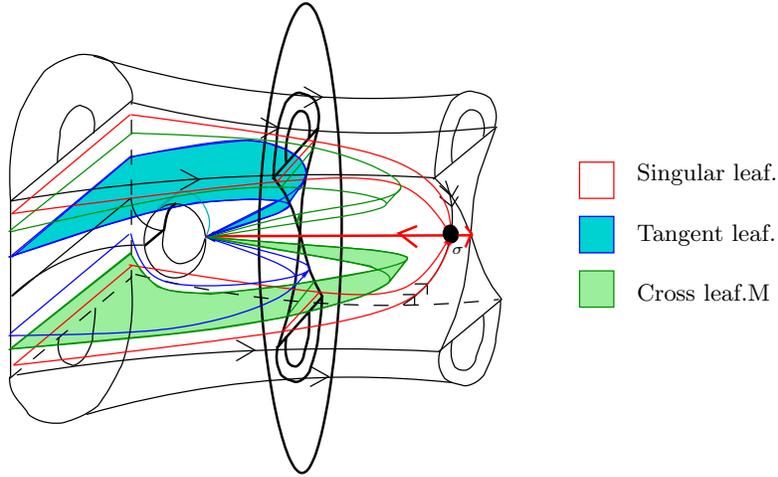
\caption{\label{f1} Flow through of the face $1$.}
\end{center}
\end{figure}

Also, we must research the flow behavior inside to the face $1$,
but in the complement of Cherry box flow. However, we can to 
observe that the behavior flow is extended to the whole circle. 
This finishes the first step.

We must to observe the flow behavior on the face $2$. In this case, 
is easy to verify that all trajectories are crossing inward to the face $2$.
Thus, the flow through of the two faces is depicted in the following figure

\begin{figure}[htv]
\begin{center}
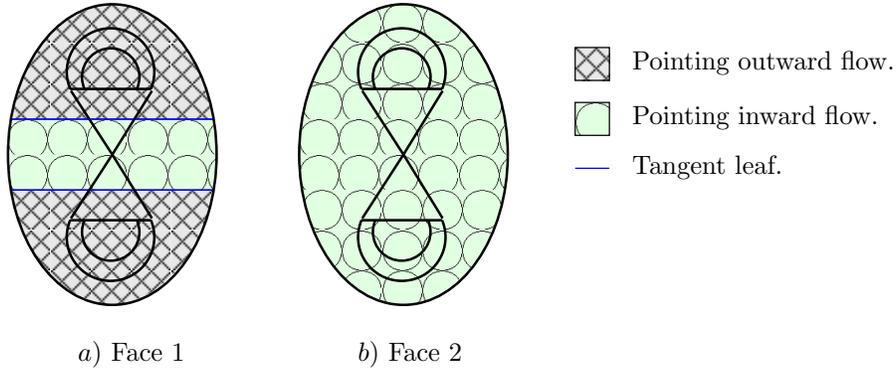
\caption{\label{f} Direction of flow through the faces.}
\end{center}
\end{figure}

Now, we construct a plug $Y$ containing a singularity $\sigma_2$.  
Consequently, the dynamical system can be transferred by means of
plug $Y$ surgery from one bitorus onto another manifold exporting some of its properties.
This singularity generates a hole and this in turns generates a solid tritorus $ST_2$ in a
way that $Y_t(ST_2)\subset Int(ST_2)$ for all $t > 0$ and $Y$ is transverse to the boundary
tritorus. The flow is obtained gluing the plugs \ref{c4}, \ref{5} with the plug $Y$ as indicated in
Figure \ref{py}. Indeed, the third hole is generated by the unstable manifold of the 
singularity $\sigma_2$.

\begin{figure}[htv]
\begin{center}
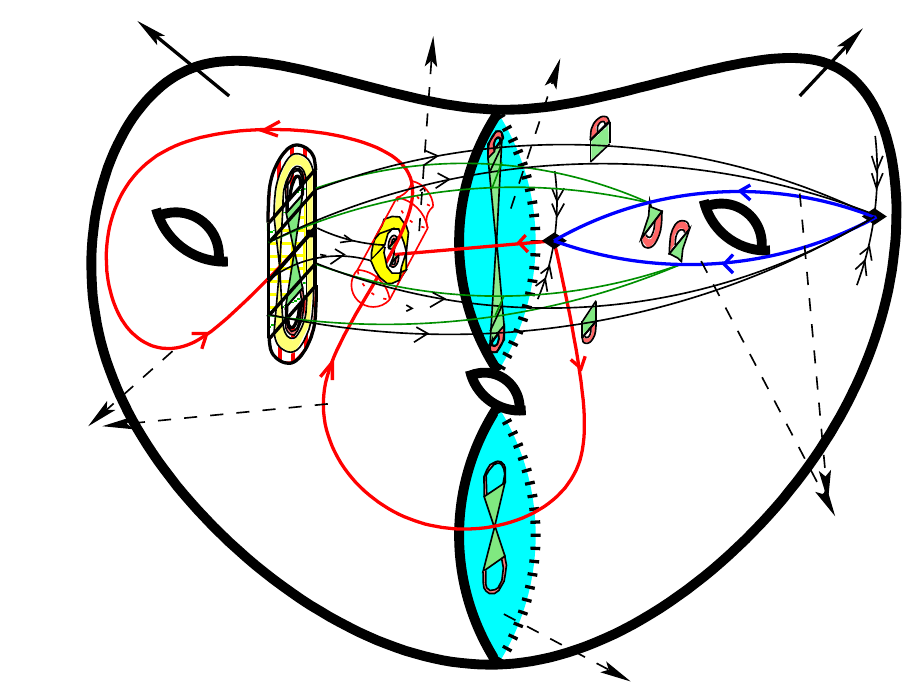
\caption{\label{py} Plug $Y$}
\end{center}
\end{figure}

In the same way from the previous subsection, we require some hypotheses for the ambient manifold
(after of gluing).

\textbf{($\hat{X}$1)}: There are two repelling periodic orbits $O_1$, $O_2$ in $Int(ST_2)$ crossing the
holes of $R$.

\textbf{($\hat{X}$2)}: There are two solid tori neighborhoods $V_1, V_2 \subset Int(ST_2)$ of $O_1,O_2$ with
boundaries transverse to $Y_{t}$ such that if $N = ST_2 \setminus (V_1 \cup V_2)$, then $N$ is a
compact neighborhood with smooth boundary transverse to $Y_{t}$ and $Y_{t}(N)\subset
N$ for $t>0$. As $N$ is a solid tritorus with two solid tori removed, we have that
$N$ is connected. 

\textbf{($\hat{X}$3)}: $R \subset N$ and the return map $H$ induced by $Y$ in $R$ satisfies the properties
\textbf{(G1)-(G3)} in Section \ref{map2}. Moreover,

$$\{q\in N : Y_{t}(q) \notin R, \forall t\in \re\} = Cl(W_Y^{uu}(\sigma_2)).$$

\bigskip

Now, we define 
$$ \hat{A}^+=Cl\left(\bigcup_{t\in\re} Y_{t}(A^+_G)\right)\qquad 
\text{and}\qquad \hat{A}^-=Cl\left(\bigcup_{t\in\re} Y_{t}(A^-_G)\right).$$

By using the Propositions \ref{prop1}, \ref{secans} and Lemma \ref{homcla} 
we can obtain that the intersection of homoclinic classes is 
the closure of the unstable manifold of two singularities.\\

\textbf{Proof of Theorem \ref{thG}}.

By using the Lemma \ref{homcla} and the Proposition \ref{secans} we have
that $Y$ is a sectional Anosov flow and $N(Y)$ is the union of two
homoclinic classes $\mathcal{H}_Y^1, \mathcal{H}_Y^2$, where $\mathcal{H}_Y^1=\hat{A}^+$
and $\mathcal{H}_Y^2=\hat{A}^-$.
It implies that $\mathcal{H}_Y^1\cap \mathcal{H}_Y^2=Cl(W_Y^u(\sigma_1)\cup W_Y^u(\sigma_2))$.
In particular $Y$ is a Venice mask, and by construction it has two singularities.




\flushleft
A. M. L\'opez B.\\
Instituto de Matem\'atica, Universidade Federal do Rio de Janeiro\\
Rio de Janeiro, Brazil\\
E-mail: barragan@im.ufrj.br

\flushleft
H. M. S\'anchez S.\\
Instituto de Matem\'atica, Universidade Federal do Rio de Janeiro\\
Rio de Janeiro, Brazil\\
E-mail: hmsanchezs@unal.edu.co

\end{document}